\documentclass{article}
\usepackage{a4}
\usepackage{amsthm,amsmath,amssymb}
\usepackage{verbatim}
\usepackage{mathrsfs}
\usepackage{stmaryrd}
\usepackage{braket}
\usepackage{proof}
\usepackage[all]{xy}
\newdir{ >}{{}*!/-4.5pt/@{>}}
\newdir{|>}{!/4.5pt/@{|}*:(1,-.2)@^{>}*:(1,+.2)@_{>}}
\usepackage{cite}

\title{Regular Functors and Relative Realizability Categories}
\author{Wouter Pieter Stekelenburg\\ Department of Mathematics\\ Utrecht University}

\theoremstyle{plain}
\newtheorem{theorem}{Theorem}
\newtheorem*{theorem*}{Theorem}
\newtheorem{lemma}[theorem]{Lemma}
\newtheorem*{lemma*}{Lemma}
\newtheorem{corollary}[theorem]{Corollary}

\theoremstyle{definition}
\newtheorem{definition}[theorem]{Definition}
\newtheorem{example}[theorem]{Example}
\newtheorem{remark}[theorem]{Remark}

\newcommand{\cat}{\mathcal}
\newcommand{\Cat}{\mathsf}

\newcommand{\Asm}{\Cat{Asm}}
\newcommand{\RT}{\Cat{RT}}

\newcommand{\Sub}{\Cat{Sub}}
\newcommand{\Sh}{\Cat{Sh}}
\newcommand{\pow}{\mathbf P}

\newcommand{\true}{\mathsf T}
\newcommand{\false}{\mathsf F}
\newcommand{\pair}{\mathsf P}

\newcommand{\dom}{\mathrm{dom}}
\newcommand{\cod}{\mathrm{cod}}
\newcommand\id{\mathrm{id}}

\newcommand{\N}{\mathbb N}

\newcommand{\converges}{\mathord{\downarrow}}
\newcommand{\convergesto}{\downarrow}
\newcommand{\A}{\mathring A}

\newcommand{\db}[1]{\llbracket #1 \rrbracket}

\newcommand{\partar}{\rightharpoonup}

\newcommand{\im}[1]{\exists_{#1}}\newcommand{\pre}[1]{#1^{-1}}

\newcommand{\Dom}{\mathrm D}

\newcommand\hide[1]{}

\begin{document}
\maketitle
\begin{abstract} The \emph{relative realizability toposes} that Awodey, Birkedal and Scott introduced in \cite{MR1909030} satisfy a universal property that involves regular functors to other categories. We use this universal property to define what relative realizability categories are, when based on other categories than of the topos of sets. This paper explains the property and gives a construction for relative realizability categories that works for arbitrary base Heyting categories. The universal property shows us some new geometric morphisms to relative realizability toposes too.
\end{abstract}

\section{Introduction} 
This paper concerns the \emph{relative realizability toposes} that Awodey, Birkedal and Scott introduced in \cite{MR1909030}. Just like realizability toposes, relative realizability toposes implicitly assign subsets of a partial combinatory algebra to the propositions of their internal languages. The members of these subsets are said to \emph{realize} the propositions they are assigned to. While realizability toposes satisfy every proposition that has an inhabited set of realizers, relative realizability toposes only satisfy propositions whose set of realizers intersect a suitable subset of the partial combinatory algebra.

Relative realizability toposes have a universal property that dictates the behavior of regular functors into other categories. Using this universal property we develop relative realizability categories for \emph{order partial combinatory algebras} (see subsection \ref{OPCA pairs}) that live in arbitrary Heyting categories. We also consider new relative realizability toposes and geometric morphisms that result from the generalized construction.

\subsection{Realizability toposes}\hide{Iets met andermans bijdragen dus}
This paper builds on the following research in realizability and category theory.

The use of toposes to study realizability started with Hylands \emph{effective topos} \cite{MR717245}. The construction of this topos is easily generalizes to other partial combinatory algebras using \emph{tripos theory} (see \cite{Pittsthesis}, \cite{MR578267}), and this is where realizability toposes come from.

In his thesis \cite{RTnLS} Longley  defined \emph{applicative morphisms} between partial combinatory algebras and proved an equivalence between these morphisms and certain regular functors between realizability toposes. He also showed that realizability toposes satisfy a universal property relative to one of their small subcategories (the category of \emph{modest sets}).

Carboni, Freyd and Scedrov \cite{MR948482} showed that every realizability topos is an \emph{exact completion} (definition in subsection \ref{contop}) of its regular subcategory of $\neg\neg$-separated objects, also called \emph{assemblies}. Menni derived conditions that make an exact completion a topos \cite{Menni00exactcompletions}, \cite{MR1900904}. 

In \cite{MR1981211} van Oosten and Hofstra introduce order partial combinatory algebras, and the realizability toposes related to them. They generalize Longley's applicative morphisms and characterize the applicative morphisms that correspond to geometric morphisms between toposes. 

We see two generalizations of realizability coming together in relative realizability, namely, a more flexible definition of validity, and the idea of developing realizability in a non classical context. 

Kleene and Vesley proposed an early example of relative realizability in \cite{MR0176922}. The partial combinatory algebra is \emph{Kleene's second model} (example \ref{Ktwo} below), whose members are functions $\N\to\N$; however, only propositions realized by total recursive function are valid. The relative realizability of Kleene and Vesley already had an intuitionist context, i.e., it considered how to prove the relative realizability of propositions constructively.

Awodey, Birkedal and Scott introduced toposes for this variation of realizability in \cite{MR1909030} and Bauer and Birkedal studied the abstract properties of relative realizability toposes in their Ph. D. theses \cite{MR1769604}, \cite{MR2701797}. In \cite{MR1933398} van Oosten and Birkedal described relative realizability as realizability over a internal partial combinatory algebra in another topos.

\subsection{In this paper}
Next section generalizes the construction of the category of assemblies so that it works for order partial combinatory algebras that live in arbitrary Heyting categories. This is a categorical way to develop relative realizability in non predicative constructive contexts. However, we use a universal property to define these categories, and introduce the construction as a constructive existence proof.

Section \ref{RRT} shows that the exact completion of a category of assemblies constructed for an order partial combinatory algebra in a topos, is a topos. It shows, in other words, that relative realizability toposes are toposes, even if we work with another base topos than the topos of sets.

We use the universal property to find regular functors form relative realizability toposes into other categories in section \ref{F}. In particular, we look at geometric morphisms from localic toposes and from other realizability toposes into realizability toposes.

\section{Relative Realizability Categories}
This section defines relative realizability categories by a universal property, and then proves the existence of categories that satisfy this property. It starts by introducing the structure of the object of realizers, and ends by considering the advantages of projective terminal objects.

Though we are mainly interested in the development of realizability in toposes, the fact that toposes have power objects does not play an essential role in realizability. We therefore develop relative realizability in the larger class of \emph{Heyting categories}. A Heyting category is a category that has first order intuitionistic logic as its internal language. Specifically, $\cat E$ is a Heyting category if for every object $X$ the class $\Sub(X)$ of subobjects of $X$ is a Heyting algebra and for every arrow $f:X\to Y$ the inverse image map $\pre f:\Sub(Y)\to\Sub(X)$ has both adjoints $\im f \dashv \pre f \dashv \forall_f$. We will often use the internal language to define objects of Heyting categories.

\subsection{OPCA pairs}\label{OPCA pairs}
In this subsection we will define \emph{ordered partial combinatory algebras} as combinatory complete ordered partial applicative structures.

\begin{definition} An \emph{ordered partial applicative structure} or \emph{OPAS} is an object with an ordering $\leq$ and a monotone partial binary operator $(x,y)\mapsto xy$ called \emph{application}, whose domain is downward closed. If $x$, $y$ and $z$ are elements of an OPAS, we write $xy\convergesto z$ for: `the application of $x$ to $y$ is defined and is equal to $z$'. 
The formula $xy\converges$ means that there is a $z$ such that $xy\convergesto z$, i.e, that $(x,y)$ is in the domain of the application operator.
\end{definition}

We single out certain partial monotone arrows of OPASes.

\begin{definition} For each OPAS $A$, $a\in A$, $n\in N$, $U\subseteq A^n$ and $f: U\to A$, we say that $a$ \emph{represents} or \emph{realizes} $f$ or that $f$ is \emph{representable}, if for all $\vec x\in \dom f$, there is a $y\leq f(\vec x)$ such that $((ax_1)\dots)x_n\convergesto y$. We call such arrows \emph{partial representable arrows}. \end{definition}

\begin{remark}\label{internal}
We interpret this definition in the internal language of the Heyting category. So relative to an OPAS $A$ a partial morphism $f: U\subseteq A^n \partar A$ is representable if and only if the following subobject of $A$ is inhabited.
\[ \db f = \Set{a\in A| \forall \vec x\in U.\exists y\in A. y\leq f(\vec x)\land ((ax_1)\dots)x_n \convergesto y } \]
This \emph{object of realizers} of $f$ may not have any global section. \end{remark}

We are interested in OPASes that represent all partial arrows that are constructed by repeated use of application.

\begin{definition} The set of \emph{partial combinatory arrows} is the least set of partial arbitrary-ary arrows $A^n\partar A$ that contains projections $\vec x\mapsto x_i$ and is closed under pointwise application. So $(x,y) \mapsto x$ and $(x,y,z) \mapsto xz(yz)$ are both examples of partial combinatory arrows. An OPAS is \emph{combinatory complete}, if every partial combinatory arrow is representable. Combinatory complete OPASes are called \emph{ordered partial combinatory algebras} or \emph{OPCAs} \cite{MR1981211}.
\emph{Partial combinatory algebras} or \emph{PCAs} are OPCAs that have the discrete ordering.
\end{definition}

Although combinatory completeness uses universal quantification in its definition, there is a way to formalize this property using only regular logic.

\begin{lemma} There is a regular theory whose models are OPCAs. \end{lemma}

\begin{proof} We easily translate our own definition of OPASes in a regular theory. We use a binary relation $\leq$ and one ternary relation $\alpha$, but write $xy\convergesto z$ instead of $\alpha(x,y,z)$.
\begin{align*}
&\vdash a\leq a & a\leq b, c\leq a &\vdash c\leq b \\ 
a\leq b, b\leq a &\vdash a=b & ab\convergesto c\land ab\convergesto d  &\vdash c=d \\
a\leq b, c\leq d, bd\convergesto e &\vdash \exists f. ac\convergesto f\land f\leq e
\end{align*}
We express that the OPAS represents a partial combinatory function $f:A^n\partar A$ by extending this theory as follows. We add a predicate $F$ to our language, and an axiom that say it is inhabited: $\vdash \exists x.F(x)$. We add a list of axioms to say that if $F(a)$, then $((ax_1)\dots)x_n \convergesto y$ for some $y\leq f(\vec x)$:
\begin{align*}
F(a) &\vdash \exists y_1. ax_1\convergesto y_1\\
F(a), ax_1\convergesto y_1 &\vdash \exists y_2. y_1x_2\convergesto y_2 \\
&\ \vdots \\
F(a), ax_1\convergesto y_1, y_1x_2\convergesto y_2,\dotsm &\vdash \exists y_n. y_n\leq f(\vec x) \land y_{n-1}x_n\convergesto y_n
\end{align*}

We can do this for each partial combinatory function and get a recursively enumerable theory; we can also use the $k,s$-basis of combinatory logic (see \cite{MR0409137}) to get an equivalent finitely axiomatized regular theory. Either way, a model $A$ for these axioms is an OPAS that represents all partial combinatory functions and therefore an OPCA.
\end{proof}

\begin{remark} Every OPCA is a model for this theory, but not always in a unique way. For each partial combinatory $f$ we may interpreted the related predicate $F$ as any inhabited subobject of $\db f$. \end{remark}

\begin{corollary} Regular functors preserve OPCAs. \end{corollary}

OPCAs are models for computation. We can view an OPCA $A$ as the set of codes for programs in a functional programming language. The application operator represents the execution of one program on the code of another. For relative realizability we want to apply a limited set of programs to a larger set of codes. This lead to the following generalization.

\begin{definition} An OPCA pair $(A',A)$ is a pair of OPASes, where
\begin{itemize}
\item $A'$ is a subobject of $A$, and application of $A'$ is the restriction of application in $A$
\item $A'$ is closed under the application in $A$. So if $x,y\in A'$ and there is a $z\in A$ such that $xy\convergesto z$, then $z\in A'$.
\item All partial combinatory arrows of $A$ are representable in $A'$. So if $f:U\subseteq A^n \to A$ is combinatory, then $\db f$ intersects $A'$.
\end{itemize}
\end{definition}

Note that if $(A',A)$ is an OPCA pair both $A'$ and $A$ are OPCAs themselves. Also, the last condition is equivalent to the condition that the sets of realizers for the partial combinatory arrows $(x,y)\mapsto x$ and $(x,y,z)\mapsto (xz)(yz)$ intersect $A'$, for reasons outlined \cite{MR0409137}. Finally, if $A$ is an OPCA, then $(A,A)$ is an OPCA pair. For this reason `absolute' realizability is a special case of relative realizability.

\begin{example}[Kleene's second model] \label{Ktwo}
There is a universal partial continuous function $\N^\N\times \N^\N \to \N^\N$ for the product topology on $\N^\N$. With this function $\N^\N$ is a PCA $\mathcal K_2$. The total recursive functions form a subPCA $\mathcal K_2^{\rm rec}$ and $(\mathcal K_2^{\rm rec},\mathcal K_2)$ is an OPCA pair. As we will see form the definition in the next paragraph, this OPCA pair exists in every topos with natural number object.

A partial continuous function $\phi:\N^\N\partar \N$ has an $i\in \N$ for each $x\in\dom \phi$ such that $f(x)=f(x')$ whenever $x'\in\N^N$ and $x_j=x'_j$ for all $j<i$; i.e., these continuous functionals only use a initial segment to determine their output. Form a bijection $\beta$ between $\N$ and the set of finite sequences of natural numbers $\N^*$, we can construct a surjection$\N^\N$ to the set of partial continuous functions $\N^\N\partar\N$. For each $y\in \N^\N$ we let $\phi_y(x)=k$ if $y(\beta^{-1}(x_0,\dotsc, x_j)) = k+1$ for the least $j\in \N$ such that $y(\beta^{-1}(x_0,\dotsc, x_j))>0$, and otherwise undefined. Using the bijection $\N^\N \simeq \N^\N\times \N$, we also get a surjection $\N^\N$ to $\N^\N\partar \N^\N$, a \emph{universal partial continuous function}. When the bijection $\beta$ between $\N$ and $\N^*$ is a recursive function, $\mathcal K_2^{\rm rec}$ is closed under application, and it represents all partial combinatory functions.
\end{example}

\begin{example}[OPCAs of downsets] If $A$ is an OPCA pair, let $\partial A$ be the set of \emph{downsets}, i.e., downward closed subobjects, of $A$. Inclusions order $\partial A$, and $\partial A$ has a partial application operator that satisfies $UV\downarrow W$ if for all $x\in U$ and $y\in V$, $xy\converges$ and if for all $z\in W$ there are $x\in U$ and $y\in V$ such that $xy\convergesto z$. In fact $\partial A$ is a new OPCA. This construction motivates the generalization from PCAs to OPCAs. Like Kleene's second model, this construction is available in any topos.
\end{example}

\begin{example} Another construction of OPCAs uses the regular functors $\Cat{Set}/I\to\Cat{Set}$ that come from \emph{filter quotient constructions} for filters on the set $I$. Because regular functors preserve OPCAs, filter quotients of $I$-indexed families of OPCAs are OPCAs. This construction does not generalize easily to other toposes.
\end{example}

\subsection{Regular Models}
The construction of the realizability toposes solves the following problem. For each Heyting category $\cat E$ and each OPCA pair $(A',A)$ in $\cat E$, we would like to construct a slightly larger category $\cat E[\A]$, where $\A$ is a subOPCA of $A$, that is closed under all partial operators $A^n\partar A$ that are realized by members of $A'$, but not closed under other partial operators. We approach this problem with two dimensional category theory. A \emph{pseudoinitial object} in a 2-category is an object for which there is an up to isomorphism unique arrow to every other object. We construct a 2-category of suitable functors $F:\cat E \to\cat C$ and subobjects $C\subseteq FA$, such that a pseudoinitial object in the 2-category should be like $\cat E[\A]$.

\begin{definition} Let $\cat E$ be a Heyting category, $(A',A)$ an OPCA pair in $\cat E$, $\cat C$ a regular category and $F:\cat E \to \cat C$ a regular functor.  An \emph{$F$-filter} is a subobject $C\leq FA$ that satisfies:
\begin{itemize}
\item If $x\in C$ and $x\leq y$, then $y\in C$.
\item If $x,y\in C$ and $xy\convergesto z$ for some $z\in FA$, then $z\in C$.
\item If $U\subseteq A$ intersects $A'$, then $FU$ intersects $C$.
\end{itemize}
\end{definition}

Regular functor preserve OPAS and filters, because their definition involves only commutative diagrams, pullbacks and images. This also means that for each pair of regular functors $F:\cat E\to\cat C$ and $G:\cat C\to \cat D$ and each $F$-filter $C$ the object $GC$ is a $GF$-filter.

\begin{definition} Let a \emph{regular model} for $(A',A)$ be a regular functor $F:\cat E\to \cat C$ with an $F$-filter. For each regular $G: \cat E\to \cat D$, each $F$-filter $C$ and each $G$-filter $D$ a morphism $(F,C) \to (G:\cat E\to\cat D,D)$ is a regular functor $H:\cat C\to\cat D$ with an isomorphism $\eta:HF\to G$, such that $\eta_A:HFA\to GA$ restricts to an isomorphism between $FC$ and $D$. A \emph{regular relative realizability category} for the pair $(A',A)$ is a \emph{pseudoinitial} regular model i.e.: there is an up to isomorphism unique regular functor from a regular relative realizability category to any regular model. \end{definition}

\begin{remark}
For each regular model $(F,C)$, $n\in\N$ and $U\subseteq A^n$ the set of partial arrows $FU\cap C^n\to C$ contains the images of partial $A'$-representable arrows $U\to A$. In that sense it is a model of the regular theory of a subset of $A$ that is closed under a set of partial operators.
\end{remark}

\begin{theorem} \label{main} There is a pseudoinitial regular model for every OPCA pair in every Heyting category. \end{theorem}

In the next couple of subsections, we define a category, a functor and a filter, and prove that these form a pseudoinitial regular model.

\subsection{Assemblies} 
This subsection explains the construction and some properties of the \emph{category of assemblies} $\Asm(A',A)$ for an OPCA pair $(A',A)$ in some Heyting category $\cat E$, the construction of a functor $\nabla:\cat E\to\Asm(A',A)$ and a $\nabla$-filter $\A$ that form a pseudoinitial regular model together.

\begin{definition}
An \emph{assembly} is a pair $(X,Y)$ where $X\in \cat E$ and where $Y$ is a subobject of $A\times X$, such that 
\begin{itemize}
\item for all $x\in X$ there is an $a\in A$ such that $(a,x)\in Y$;
\item if $(a,x)\in Y$ and $b\leq a$, then $(b,x)\in Y$.
\end{itemize}

For each pair of assemblies $(X,Y)$ and $(X',Y')$, and each $f:X\to X'$, a $V\subseteq A$ \emph{tracks} or \emph{realizes} $f$ if for all $a\in V$ and $b\in A$, if $(b,x)\in Y$ then $ab\converges$ and $(ab,f(x))\in Y'$. A \emph{morphism} $(X,Y) \to (X',Y')$ is an arrow $f: X \to X'$ for which there exists a subobject $V$ of $A$ that intersects $A'$ and that tracks $f$.

We summarize this by saying the following diagram must commute.
\[ \xymatrix{
V\times Y \ar[d]_{(v,y)\mapsto y} \ar[dr]^{\quad (v,a,x)\mapsto (va,f(x))} \\
Y\ar[d]_{(a,x)\mapsto x} & Y'\ar[d]^{(a,x)\mapsto x} \\
X \ar[r]_f & X'
} \]
\end{definition}

\begin{remark} When developing realizability in a topos, we can present assemblies as $\pow A$-valued functions. We work with binary relations instead, so that we can apply to construction to Heyting categories where $\pow A$ does not exist. \end{remark}

We will prove that the category of assemblies is a Heyting category, after we introduce some extra structure that will help us to do so.

\newcommand\down{\mathord\downarrow}
\begin{remark} Note that if $V$ tracks a morphism, then so does $\down V = \set{a\in A| \exists v\in V.a\leq v}$. 
\end{remark}

\begin{lemma} Assemblies and morphisms form a category. \end{lemma}

\newcommand\comb{\mathsf}
\begin{proof} 
Let $\comb I = \db{x\mapsto x}$. This combinator intersects $A'$ and tracks $\id_X:(X,Y)\to(X,Y)$. Let $\comb B = \db{(x,y,z)\mapsto x(yz)}$. If $U$ and $V$ intersect $A'$, $U$ tracks $f:(X,Y) \to (X',Y')$ and $V$ tracks $g:(X',Y') \to (X'',Y'')$, then $\comb BVU$ tracks $g\circ f$. \end{proof}

\begin{definition} We denote the \emph{category of assemblies} by $\Asm(A',A)$. \end{definition}

\begin{remark}
Our definition of the category assemblies is complicated, but equivalent to the conventional definition (see \cite{MR2479466}) in the internal language of a topos. If $(X,Y)$ is an assembly, then the projection $\pi_1:Y\to X$ is a family of inhabited downsets of $X$, which can be represented as an inhabited downset valued morphism $X\to\pow A$. Our definition of morphism lets the underlying Heyting category believe there is an element of $A'$ that tracks it.
\end{remark}

\hide{
\begin{lemma} \label{better} For assemblies $(X,Y)$ and $(X',Y')$ and any arrow $f:\Dom(X,Y) \to \Dom(X',Y')$, let 
\begin{align*} &\db{f:(X,Y) \to (X',Y')} = \\ & \Set{a\in A| \forall (b,x)\in Y. \exists (c,y)\in Y'. ab\convergesto c, x\in \dom f \land f(x) = y }\end{align*}
$f$ is a morphism $(X,Y) \to (X',Y')$ if and only if $\db{f:(X,Y) \to (X',Y')}$ intersects $A'$.
\end{lemma}

\begin{proof} We define a useful family of combinatory functions by recursion:
\[ \alpha_0(x) = x \quad \alpha_{n+1}(x_0,\dots, x_{n+1}) = \alpha_n(x_0,\dots x_n)x_{n+1} \]
If $\db{f: (X,Y) \to (X',Y')}$ intersects $A'$, then $(\db{f: (X,Y) \to (X',Y')},\alpha_1)$ tracks $f$, so $f$ is a morphism.  

For each $g:(X,Y)\to (X',Y')$ there is a tracking $(V\subseteq A^n,h)$. 
We build a new tracking for $g$ that has a more suitable form. 
\[ \db h = \Set{ a\in A| \forall \vec x\in \dom h. \exists y\leq h(\vec x).\alpha_{n+1}(a,\vec x)\convergesto y } \]
$(\alpha_{n+1}, \db h\times V)$ is the new tracking.
For all $(a,\vec x)\in \db h\times V$, $b\in A$ and $(c,y)\in Y$, if $b = \alpha_n(a,\vec x)$, then $bc\converges$ and $bc = \alpha_{n+1}(a,\vec x,c) \leq h(\vec x,c)$; therefore $(bc,g(x))\in Y'$. This means that $\alpha_n: \db h\times V \to \db{g:(X,Y)\to (X',Y')}$. Because $\db h\times V$ intersects $(A')^{n+1}$ and $A'$ is closed under application, the subobject $\db{g:(X,Y)\to (X',Y')}$ intersects $A'$.
\end{proof}
}

\begin{remark} Let $\Dom(X,Y) = X$. 
$\Dom:\Asm(A',A) \to \cat E$ is a faithful functor. For an OPCA pair $(A',A)$ in the category of sets this functor is \emph{not} isomorphic to the global sections functor unless $A'=A$. For that reason we use the $\Dom$ of \emph{domain} rather than the $\Gamma$ of \emph{global section} to symbolize this functor.
\end{remark}

This category has quite a bit more structure then just any regular category.

\begin{lemma} \label{heyt} The category of assemblies is a Heyting category. \end{lemma}

\newcommand\term{\mathbf 1}
\begin{proof} 
We start with finite limits. If $\term$ is terminal, then for each assembly $(X,Y)$ the unique map $!:X\to \term$ is a morphism $(X,Y)\to (\term,A\times \term)$. To help construct pullbacks, let
\begin{align*}
\true &= \db{(x,y)\mapsto x} &
\false &= \db{(x,y)\mapsto y} &
\pair &= \db{(x,y,z)\mapsto zxy} 
\end{align*} 
Given $f:(X,F) \to (Z,H)$ and $g:(Y,G)\to (Z,H)$ let $p:W\to X$ and $q: W\to Y$ be a pullback cone for $f$ and $g$ in $\cat E$.  Then let 
\[ K = \Set{(a,w)\in A\times W|  \begin{array}{l} \forall t\in \true. at\converges, (at,pw)\in F,\\
 \forall f\in \false.af\converges, (af,qw)\in G\end{array}} \] 
Let $\comb H = \db{(x,y)\mapsto yx}$, then $\comb H\true$ tracks $p:(W,K)\to(X,F)$ and $\comb H\false$ tracks $q:(W,K)\to(Y,F)$. Therefore $p$ and $q$ form a commutative square with $f$ and $g$ in the category of assemblies.
If $L$ tracks $r:\xi \to (X,F)$ and $M$ tracks $s:\xi \to (Y,G)$ for any other assembly $\xi$, let $N = \db{(p,x,y,z) \mapsto (p(xz)(yz)}$. There exists a unique factorization $(r,s):\Dom\xi \to W$ through $p$ and $q$ and $N\pair LM$ tracks $(r,s)$. We see both that $\Asm(A',A)$ has all finite limits and that $\Dom$ preserves them.

Next: images. Given $f:(X,Y) \to (X',Y')$ let $\exists_{f}(X,Y) = (\exists_f(X),\exists_{1\times f}(Y))$. By definition $\Dom(\im f(X),\exists_{1\times f}(Y)) = \exists_f(\Dom(X,Y))$, so $\comb I$ tracks $f:(X,Y) \to \exists_{f}(X,Y)$. If $p,q: \xi \to (X,Y)$ is a kernel pair for $f$ in $\Asm(A',A)$ then it is a kernel pair for $f$ in $\cat E$, because $\Dom$ preserves finite limits. If $V$ tracks some $g:(X,Y) \to \psi$ that satisfies $g\circ p = g\circ q$, then it also tracks the factorization of $g$ through the image of $f$. Hence $\exists_f(X,Y)$ is a coequalizer for the kernel pair.

We need to show that regular epimorphisms are stable. An epimorphism $e:(X,Y) \to (X',Y')$ is regular, if $\exists_e(X,Y) \simeq (X',Y')$. Therefore, we can assume that $(X',Y') = (\exists_f(X), \im{\id\times e}(Y))$ without loss of generality. Since $\Dom((\im f(X), \exists_{\id\times e}(Y)) = \exists_e(\Dom(X,Y))$, the functor $\Dom$ preserves regular epimorphisms and is itself regular. 
For any $f:(Z,H)\to (X', Y')$, let $p: (W,K)\to (X,Y)$, $q:(W,K)\to (Z,H)$ be a pullback cone for $e$ and $f$ like the one we constructed above. The arrow $q$ is a regular epimorphism in $\cat E$, because $\Dom$ preserves pullbacks and $\cat E$ is a regular category. Furthermore, $\comb I$ tracks $e$, $\comb H\true$ tracks $p$ and $\comb H\false$ tracks $q$. 
If $V$ tracks $f$, then $\db{(p,v,x) \mapsto px(vx)}\pair V$ tracks $\id_Z: (Z,H)\to \exists_q(W,K)$, while $\comb H\false$ tracks $\id_Z:\exists_{q}(W,K) \to (Z,H)$. So $\exists_{q}(W,K)\simeq (Z,H)$ and $q$ is a regular epimorphism in $\Asm(A',A)$. So pullbacks of regular epimorphism are regular epimorphisms.

We see that $\Asm(A',A)$ is a regular category and that $\Dom:\Asm(A',A)\to\cat E$ is a regular functor. We construct joins of subobjects to show that all subobject posets are lattices. First we show how to represent a subobject of an assembly $(X',Y')$ as a subobject $F\subseteq A\times X$.

A morphisms $m:X\to Y$ is monic if and only if $\im m(X)\simeq X$. Therefore, given a mono $m:(X,Y) \to (X',Y')$ in $\Asm(A',A)$ we have $(X,Y)\simeq (\im m(X),\exists_{\id\times m}(Y))$. Let $\overline Z = \set{x\in X|\exists a\in A.(a,x)\in Z}$ for all $F\subseteq A\times X$. For each subobject $U$ of $(X,Y)$ there is an $F\subseteq A\times X$ such that $\id_{\overline F}: (\overline F,F) \to (X,Y)$ is tracked and such that $\id_{\overline F}:(\overline F,F) \to (X,Y)$ represents $U$. For all $X\in \cat E$ and all $F,G\leq A\times X$, say that $U$ tracks $F\leq G$ if it tracks $\id_{\overline F}:(\overline F,F)\to(\overline G,G)$.

On to joins. For any pair $F,G\leq A\times X$, let 
\[ F\vee G = \Set{ (a,x)\in A\times X| \begin{array}{c}
\exists b\in A,p\in\pair, t\in \true. (b,x)\in F, a\leq ptb \\ \vee\\
\exists b\in A,p\in\pair, f\in \false. (b,x)\in G, a\leq pfb 
\end{array}} \]
Now $\pair\true$ tracks $Y\leq Y\vee Y'$ and $\pair\false$ tracks $Y\leq Y\vee Y'$, therefore $Y\vee Y'$ is an upper bound of $\Set{Y,Y'}$.
If $U$ tracks $Y\leq Z$ and $U'$ tracks $Y'\leq Z$, then $\db{(t,f,u,u',a) \mapsto at[u(af)][u'(af)]} \true\false U U'$ tracks $Y\vee Y' \leq Z$. Therefore $Y\vee Y'$ is the least upper bound.

There is only one assembly $(\bot,\bot)$ over the initial object $\bot$ of $\cat E$, and it is embedded in every other assembly. This is the bottom element of the poset of subobjects over every assembly, which poset we can now call a \emph{lattice} of subobjects.

We now construct right adjoints to the inverse image maps. For each $f:(X,Y) \to (X',Y')$ and $F\leq Y$ let:
\[ \forall_f(F) = \Set{(a,y)\in A\times X' | \forall (b,x)\in Y. f(x)=y \to ab\converges\land (ab,x)\in F }\]
Pullbacks induce the inverse image map. Therefore, if $G\leq A\times X'$ represents a subobject of $(X',Y')$, then the following object represents its inverse image.
\[ f^{-1}(G) = \Set{(a,x)\in A\times X|  \begin{array}{l} \forall t\in \true. at\converges, (at,x)\in Y,\\
 \forall f'\in \false.af\converges, (af',f(x))\in G\end{array}} \]
If $U$ tracks $G \leq \forall_f(F)$, let $h(t,f,u,x) = u(xt)(xf)$ then $\db h \true\false U$ tracks $f^{-1}(G)\leq F$. If $V$ tracks $f^{-1}(G)\leq F$, 
then $\db{(p,v,x,y)\mapsto v(pxy)} \pair V$ tracks $G \leq \forall_f(F)$. We see that $\forall_f$ is right adjoint to $f^{-1}$.

We now have shown that $\Asm(A',A)$ is regular, that subobjects form a lattice and that inverse image maps have both left an right adjoints. We construct Heyting implications form these right adjoints: if $m:(X,Y) \to (X',Y')$ represent a subobject $U$ of $(X',Y')$, and $V$ is another subobject of $(X',Y')$, then $U\to V = \forall_m\pre m(V)$.

We conclude that lattices of subobjects are Heyting algebras and that the inverse image maps have both left and right adjoints. Therefore $\Asm(A',A)$ is a Heyting category.
\end{proof}

On to the functor.

\begin{definition} For each object $X$ in $A$ let $\nabla X = (X,A\times X)$. For each arrow $f:X\to Y$, let $\nabla f = f$.\end{definition}

The arrow $\nabla f$ is a morphism $\nabla X\to\nabla Y$ because $\Dom\nabla = \id_\cat E$ and $\id_A\times f:A\times X\to A\times Y$. The functor $\Dom$ is a faithful $\Asm(A',A) \to \cat E$ and $\nabla$ is a right inverse. In fact $\nabla$ is right adjoint to $\Dom$, because $\comb I$ tracks the inclusion $\id_X:(X,Y)\to\nabla\Dom(X,Y)$ for every assembly $(X,Y)$.

\begin{lemma} The functor $\nabla$ is regular. \end{lemma}

\begin{proof} In regular categories $e:X\to Y$ is a regular epimorphism if and only if $\exists_e(X)\simeq Y$. So let $e\in \cat E$ be a regular epimorphism. Images lift to $\Asm(A',A)$, and 
\[ \exists_{\nabla e}(\nabla X) = (Y, \exists_{\id_A\times e}(A\times X)) \simeq \nabla Y \]
Therefore $\nabla$ preserves regular epimorphisms. Since $\nabla$ is right adjoint to $\Dom$, it also preserves all limits. That makes it a regular functor.
\end{proof}

\hide{
\begin{remark} By the way, a regular functor between Heyting categories preserves OPCAs, because the functor inflates the object of realizers of each partial representable arrow. Let $(A',A)$ be an OPCA pair in $\cat E$ and let $F:\cat E\to\cat F$ be a regular functor into a Heyting category. For each representable $f: U\subseteq A^n\to A$, let
\[ R(f) = \Set{ (a,\vec x)\in A\times U | \exists y\leq f(\vec x).\alpha_n(a,\vec x)\convergesto y } \]
If $p: A\times U\to A$ is the first projection, then by definition $\db f = \forall_p(R(f))$ and that implies $\pre p(\db f) \subseteq R(f)$.
Just because $F$ preserves finite limits, it preserves OPASes and partial combinatory arrows. For that reason the definition of $R$ makes sense relative to $FA$. But $F$ also preserves regular epimorphisms and this implies $FR(f) = R(Ff)$. From $\pre p(\db f) \subseteq R(f)$ we now deduce $F\db f \subseteq \db{Ff}$.
If $f$ is any partial combinatory arrow, then $\db f$ intersects $A'$. Therefore $\db{Ff}$ intersects $FA'$. So for every partial combinatory arrow $f:FA^n \partar FA$, $\db f$ intersects $FA'$ and this makes $(FA',FA)$ an OPCA pair.
\end{remark}
}

We have a category and we have a functor. Now we need a filter, which is some subobject of $\nabla A$.

\begin{lemma} Let $\set\leq$ be $\Set{ (x,y)\in A^2 | x\leq y }$ and let $\A = (A,\set\leq)$. The identity map $\id_A:\A\to\nabla A$ is a monomorphism that represents a filter on $\A$. \end{lemma}

\begin{proof} That $\id_A$ is a morphism follows from the fact that $\id_A\times \id_A:\set\leq \to A^2$ is just the inclusion. If $f,g:(X,Y)\to \A$ satisfy $\id_A\circ f= \id_A\circ g$ then $f=g$, so $\id_A$ is a monomorphism, and monomorphisms represent subobjects.

Because $\nabla$ is regular $\nabla\set\leq$ is a partial ordering of $\nabla A$. Relative to this ordering $\A$ is an upward closed subobject. The order $\set\leq$ has two projections $\set\leq \to A$. By pulling $\A$ back along the first projection we get the object of pairs of element of $\nabla A$, where the first is some element $\A$ and the second is a greater element of $\nabla A$. The subobject $\comb I$ tracks the second projection of this pullback to $\A$. This shows $\A$ is upward closed under the ordering $\nabla\set\leq$.

If $U$ is a subobject of $A$ that intersects $A'$, then $\A$ intersects $\nabla U$. This means the the support of the pullback of the inclusions of $\A$ and $\nabla U$ is a terminal object. Using the constructions in the proof of lemma \ref{heyt} we find that the assembly $(\term,\down U)$ represents this support.

Let $\comb K = \db{(x,y)\mapsto x}$. The unique arrow $\id_\term:\term \to\term$ is a morphism $(\term, A)\to (\term,\down U)$, because $U$ intersects $A'$ and $\comb K U$ tracks it. Therefore $\A\land \nabla U$ is inhabited and $\A$ intersects $\nabla U$ if $A'$ intersects $U$.

Let $D\subseteq A^2$ be the domain of the application operator. We intersect $\A\times \A$ with $\nabla D$ by pulling back along the inclusion $\id_D: D \to A^2$. To get a simpler representation, we project down along the inclusion of $(\A\times \A)\cap \nabla D$. This way $\A^2\cap \nabla D \simeq (D, E)$, where
\[ E = \Set{(a,b,c)\in A\times D| \forall t\in \true, f\in \false. at=b, af=c } \] 
Let $\comb G = \db{(t,f,x)\mapsto (xt)(xf)}$. The application operator $\alpha_1:D\to A$ is a morphism $(D,E) \to \A$ because $\comb{GTF}$ tracks it. This means in the internal language of $\Asm(A',A)$ that if $x,y\in \A$ and $xy\converges$, then $xy\in \A$.

The assembly $\A$ is a filter because it is downward closed, it intersects $\nabla U$ when $A'$ intersects $U$ and it is closed under application.
\end{proof}

We have a category $\Asm(A',A)$, a regular functor $\nabla:\cat E\to\Asm(A',A)$ and a $\nabla$-filter $\A\leq \nabla A$, so we have a regular model for $(A',A)$. If this model is pseudoinitial, the common structure of all regular models generates every object and morphism: the base category, images and preimages and the filter. We show this in the next couple of lemmas.

\begin{lemma} \label{gen1} For each assembly $(X,Y)$ let $a:Y\to A$ be the first projection and $x:Y\to X$ be the second projection. 
\[ (X,Y) \simeq \exists_{\nabla x} ((\nabla a)^{-1}(\A)) \]
\end{lemma}

We can compute this using the constructions for pullbacks an images given in the proof of lemma \ref{heyt}.

A more traditional way to state this lemma is as follows.
\begin{definition} An assembly $(X,Y)$ is \emph{partitioned} if there is an arrow $f:X\to A$ in $\cat E$ such that 
\[ (X,Y)\simeq (\nabla f)^{-1}(\A) \]
\end{definition}

\begin{lemma} Partitioned assemblies cover all assemblies. \end{lemma}

We will refer to regular epimorphisms from partitioned assemblies to other assemblies as \emph{partitioned covers}. 

\begin{remark} While partitioned assemblies are projective objects in realizability categories over the category of sets and other categories where epimorphisms split, this does not generalize to all toposes, let alone all Heyting categories. \end{remark}

The structure of regular models also generates the class of morphisms of $\Asm(A',A)$. The proof of the following lemma reveals how our definition of morphism works.

\begin{lemma} \label{gen2} Each morphism $f:(X,Y)\to (X',Y')$ is the unique factorization of $\nabla\Dom f$ composed with $\id_X: (X,Y)\to\nabla\Dom(X,Y)$ through $\id_{X'}: (X',Y')\to\nabla\Dom(X',Y')$.
\end{lemma}

\begin{proof}
Let $U$ track $f:(X,Y)\to (X',Y')$. According to the definition of morphisms the following diagram commutes and the vertical arrows are regular epimorphisms.
\[ \xymatrix{ U\times Y \ar[dr]^{ (u,a,x)\mapsto (ux,f(x))} \ar@{-|>}[d]_{(u,y)\mapsto y} \\
Y \ar@{-|>}[d]_{(a,x)\mapsto x} & Y' \ar@{-|>}[d]^{(a,x)\mapsto x} \\
X \ar[r]_f & X' }\]

We will use the internal language here to define some pullbacks. Let
\begin{align*}
P &= \Set{(u,(a,x))\in \nabla(U\times Y) | (u,a)\in \A^2 } \\
P' &= \Set{(a,x)\in Y'| a\in \A }
\end{align*}
The assembly $P$ covers $(X,Y)$. The assemblies $\im{\nabla((u,(a,x))\mapsto x)}(P)$ and $\im{(a,x)\mapsto x}(Y)$ are the same subobject of $\nabla X$, because  $\nabla U$ intersects $\A$. The restriction of $(u,a,x)\mapsto (ux,f(x))$ to $P$ lands in $P'$, because $\A$ in closed under application. And $(X',Y') = \im{\nabla((a,x)\mapsto x)}$ by lemma \ref{gen1}.

Consider the following diagram.
\[\xymatrix{ 
P \ar@{-|>}[d]_{(u,(a,x))\mapsto x} \ar[r]^{(u,a,x)\mapsto (ux,f(x))} & P' \ar@{-|>}[d]^{(a,x)\mapsto x} \\
(X,Y) \ar[d]_{\id_X} \ar[r]_f & (X',Y') \ar[d]_{\id_{X'}} \\
\nabla\Dom(X,Y) \ar[r]_{\nabla\Dom f} & \nabla\Dom(X,Y)
}\]
We just proved that the outer square commutes and the lower square commutes by the definition of morphism. The upper square commutes because $\id_{X'}$ is monic.

Conclusion: each morphism of assemblies $f:(X,Y)\to(X',Y')$ equals the unique factorization of $\nabla\Dom f\circ \id_{X}$ over $\id_{X'}$.
\end{proof}

With these lemmas in hand, we can prove that $\Asm(A',A)$ is a pseudoinitial regular model.

\subsection{Existence Theorem}
We this section we will show that $(\nabla,\A)$ is a pseudoinitial regular model. Thus we prove theorem \ref{main}.

\begin{theorem} There is a pseudoinitial model for every OPCA pair in every Heyting category. \end{theorem}

\begin{proof} Given a regular model $(F,C)$ for an OPCA pair $(A',A)$, we choose an object map $F_C$. For each assembly $(X,Y)$, let $a:Y\to A$ and $x:Y\to X$ be the projections. Let $F_C(X,Y)$ be isomorphic to $\im{Fx}(\pre{Fa}(C))$. By definition $U(X,Y) = \exists_x(Y)$, therefore $FU(X,Y) = \exists_{Fx}(Y)$ and $F_C(X,Y)$ is a subobject of $FU(X,Y)$.

While the object map requires a strong form of choice or a small category $\cat E$, once we have this map, there is a unique way to extend it to a functor, thanks to lemma \ref{gen2}. If $U$ tracks $f:(X,Y)\to (X',Y')$, then the following square commutes, and the vertical arrows are epic because $F$ is a regular functor.
\[ \xymatrix{
F(U\times Y) \ar[rrr]^(.55){F((u,a,x)\mapsto(ua,f(x)))} \ar@{-|>}[d]_{(u,(a,x))\mapsto x} &&& FY'\ar@{-|>}[d]^{(a,x)\mapsto x} \\
FX \ar[rrr]_{Ff} &&& FX'
}\]
Because $C$ is a filter, the subobject $\Set{(u,(a,x))\in F(U\times Y) | (u,a)\in C^2 }$ covers $F_C(X,Y)$ and the restriction of $F(g\times f)$ factors through $\Set{(a,x)\in FY| a\in C}$, the subobject of $FY'$ that covers $F_C(X,Y)$. Therefore there is a unique factorization through $(X',Y')$ of $F\Dom f$ restricted to $(X,Y)$. We define $F_Cf$ to be that morphism.

This functor preserves images and preimages by definition and therefore is regular. Also $F_C\nabla X \simeq FX$ and $F_C\A \simeq C$, so this regular functor is a morphism of regular models.

Every regular $G:\Asm(A',A) \to \cod F$ such that $G\A\simeq C$ and $G\nabla\simeq F$ is isomorphic to $F_C$. Pullbacks preserve the isomorphism $F_C(\A)\simeq G\A$, so that the functors have to agree on all partitioned assemblies. The isomorphism $F_C\nabla\simeq G\nabla$, and the relation of each morphism $f$ to $\nabla\Dom f$ now forces the functors to agree on all assemblies.

We conclude that the functor $\nabla: \cat E \to \Asm(A',A)$ and the filter $\A\subseteq \nabla A$ together form a pseudoinitial regular model for every OPCA pair $(A',A)$ in every Heyting category $\cat E$.
\end{proof}

We take this result one step further to prove that certain categories of regular functors are equivalent to certain categories of subobjects.

\begin{definition} For every Heyting category $\cat E$, OPCA pair $(A',A)$, regular category $\cat C$ and regular functor $F:\cat E\to\cat C$, a \emph{regular extension} of $F$ is a regular functor $G:\Asm(A',A)\to \cat C$ with an isomorphism $\phi:G\nabla\to F$. A morphism of regular extensions $(G,\phi)\to(H,\psi)$ is a natural transformation $\eta:G\to H$ that commutes with the isomorphisms, i.e., $\eta\nabla\circ \phi = \psi$.
\end{definition}

\begin{corollary} For a fixed regular $F:\cat E\to \cat C$ there is an equivalence of categories between the poset of $F$-filters, whose ordering is inclusion, and the category of regular extensions of $F$.
 \end{corollary}

\begin{proof} We first show how natural transformations induce inclusions of filters.
Let $G,H:\Asm(A',A)\to\cat C$ be regular functors, let $\eta: G\to H$ and let $\eta\Delta: G\Delta \to H\Delta$ be an isomorphism of functors. Consider the following naturality square.
\[ \xymatrix{
G\A \ar[r]^{\eta_{\A}} \ar[d]_{\id_A} & H\A \ar[d]^{\id_A}\\
G\nabla A \ar[r]_{\eta_{\nabla A}} & H\nabla A
}\]
Since $G$ and $H$ are both regular, the vertical arrows are monic and the lower arrow is an isomorphism. Therefore $\eta_{\A}$ must be monic too. If there are isomorphisms $\phi: G\nabla \to F$ and $\psi: H\nabla \to F$, and if $\eta\nabla$ commutes with these isomorphisms, then $\eta\nabla$ is an isomorphism. Hence $G\A\subseteq H\A$.

Next we construct a natural transformation from an inclusion of filters. Let $C\subseteq C'$ be $F$-filters. Pullbacks preserve the inclusion $C\subseteq C'$ and since partitioned assemblies are pullbacks, we can define $\eta_P: F_CP\to F_{C'} P$ to be this pulled back inclusion. Each assembly $X$ has a partitioned cover $e:P\to X$, which we use the construct this diagram.
\[ \xymatrix{
F_CP \ar[r]^{\eta_P} \ar@{-|>}[d]_{F_Ce} & F_{C'}P \ar@{-|>}[d]^{F_{C'}e} \\
F_C(X,Y) \ar[d]_{F_C\id_X} & F_{C'}(X,Y) \ar[d]_{F_{C'}\id_{X}} \\
F_C\nabla\Dom(X,Y) \ar[r]^{\sim} & F_{C'}\nabla\Dom(X,Y)
}\]
There is a unique arrow $F_C(X,Y)\to F_C(X',Y')$ that commutes with all the arrows in the diagram, and we define $\eta_{X,Y}$ to equal this arrows. Thus we get a natural transformation $\eta$ for which $\eta\nabla$ is an isomorphism.

The natural transformation $\eta$ we constructed in the last paragraph induces the inclusion $C\subseteq C'$. Also, the diagram above shows that any transformation that induces this inclusion must equal $\eta$. Therefore, there is an equivalence of categories between the poset of $F$-filters and the regular extensions of $F$.
\end{proof}

\subsection{Projective Terminals}
If the terminal object of the underlying Heyting category $\cat E$ is projective, e.g., in the category of sets, then every inhabited set has a global section. This simplifies the construction of the category of assemblies. Since each inhabited object has a section, global sections realize each representable arrow of each OPAS and therefore each partial combinatory function of each OPCA.

\begin{definition} Let $A$ be an OPAS. 
An arrow $f: A^n \to A$ is \emph{globally} representable if $\db f$ (see remark \ref{internal}) has a global section.
\end{definition}

\begin{lemma} In every Heyting category $\cat E$ every globally representable morphism is representable. If the terminal object is projective, any representable morphism is globally representable.
\end{lemma}

\begin{proof} Any arrow $f: U\subseteq A^n \to A$ is representable if the following object is inhabited.
\[ \db f = \Set{a\in A| \forall \vec x\in U.\exists y\in A. y\leq f(\vec x)\land ((ax_1)\dots)x_n \convergesto y } \]
If $f$ is globally representable, then $\db f$ has a global section. This makes $\db{f}$ inhabited and therefore $f$ representable. If $g$ is representable and the terminal object is projective, then $\db f$ has a global section, and this section globally represents $f$.
\end{proof}

We can use global representability to construct categories of assemblies for certain pairs of ordered partial applicative structures in categories that have finite limits, but are not necessarily regular or Heyting. In the following lemma we formulate one property global representability that lets us do this.

\begin{lemma} For every finite limit category $\cat C$ and let $\Gamma:\cat C\to \Cat{Set}$ be the global sections functor. For every ordered partial applicative structure $A\in\cat C$, a partial arrow $f:A^n\partar A$ is globally representable in $A$ if and only if $\Gamma f$ is representable in $\Gamma A$. \end{lemma}

\begin{proof} The set of realizers $\db{\Gamma f}\subseteq \Gamma A$ is inhabited precisely when $\db f$ has a global section. \end{proof}

One possible definition of a category of assemblies for a global OPCA pair is now the following.

\begin{definition} For any finite limit category $\cat C$ let $\Gamma:\cat E\to\Cat{Set}$ be the global sections functor. A pair of OPASes $A'\subseteq A$ in $\cat C$ is a \emph{global OPCA pair}, if $(\Gamma A,\Gamma A')$ is an OPCA pair. The category of assemblies for the global OPCA pair $(A',A)$ is the fibred product of $U:\Asm(\Gamma A',\Gamma A)\to \Cat{Set}$ and $\Gamma: \cat C \to \Cat{Set}$.
\end{definition}

We return to our own definition of a category of assemblies over arbitrary Heyting categories. Assuming a projective terminal object, we can simplify the definition of a morphism of assemblies.

\begin{lemma} For each morphism $f:(X,Y) \to (X',Y')$ there is a global combinator $r:\term\to A$ such that $(ra,f(x))\in Y'$ for all $(a,x)\in Y$. For every $f':X\to X'$ and every pair of assemblies $(X,Y)$ and $(X',Y')$ and each global section $r':\term\to A'$, if $r'a\converges$ and $(r'a,x)\in Y'$ for all $(a,x)\in Y$, then $f'$ is a morphism. \end{lemma}

\begin{proof} For any tracking $U$ of $f$ that intersects $A'$, there is a global section $r:\term\to U\cap A'$ that satisfies our requirements. The subobject $\set{r'}$ tracks $f':(X,Y)\to (X',Y')$, so $f'$ is a morphism.
\end{proof}

This is the definition of morphism of assemblies one finds in other sources, like \cite{MR2479466}. So known categories of assemblies are special cases of our construction. 

The category of assemblies is a pseudoinitial regular model of an OPCA pair. In the next section we will show a similar definition of relative realizability toposes.

\section{Relative Realizability Toposes}\label{RRT}
In this section we assume that the underlying category $\cat E$ is a topos. Under that condition, we can construct a topos out of the category of assemblies.

\begin{definition} For every topos $\cat E$ and every OPCA pair $(A',A)$ in $\cat E$ an \emph{exact model} is a regular functor $F$ from $\cat E$ to an exact category $\cat C$, together with an $F$-filter. A \emph{relative realizability topos} $\RT(A',A)$ is an pseudoinitial exact model.
\end{definition}

\begin{theorem} Relative realizabilty toposes exist for every OPCA pair in every topos. \end{theorem}

\begin{proof} We start with a construction that turns regular categories into exact ones. The 2-category of exact categories is a reflective subcategory of the 2-ca\-te\-gory of regular categories (see \cite{MR1358759}). This means that for every regular category $\cat C$ there is an exact category $\cat C_{\mathit{ex/reg}}$ and a regular functor $I:\cat C \to\cat C_{\mathit{ex/reg}}$ such that every regular functor from $\cat C$ to an exact category $\cat D$ factors through $I$ up to isomorphism. Categories with this property of $\cat C_{\mathit ex/reg}$ are called \emph{exact completions} of $\cat C$.

Let $\cat E$ be a topos, $\cat D$ an exact category and $F:\cat E\to \cat D$ a regular functor. If $(F,C)$ is an exact model, then  there is an up to isomorphism unique regular functor $F_C: \Asm(A',A) \to \cat D$ such that $F_C\nabla\simeq F$ and $F\A\simeq C$, because exact models are regular models. $F_C$ factors up to isomorphism through exact completions of $\Asm(A',A)$ because its codomain is exact. The regular functor $I:\Asm(A',A) \to \Asm(A',A)_{\textit{ex/reg}}$ creates an exact model $(I\nabla,I\A)$, and we see now that it is pseudoinitial.
\end{proof}

We give a construction for an exact completion of $\Asm(A',A)$ in section \ref{contop}. Before that we want to prove that relative realizability toposes are indeed toposes. We will use a result from Mat\'ias Menni's thesis \cite{Menni00exactcompletions} for this: if a regular category is locally Cartesian closed and has a generic mono, then its exact completion is a topos.

\subsection{Local Cartesian Closure}
Local Cartesian closure means Cartesian closure of all slice categories. We prove that the category of assemblies is locally Cartesian closed in two steps. Firstly we prove that if a Heyting category has a Cartesian closed reflective subcategory, then it is Cartesian closed under some conditions on the reflector. Secondly we prove that for each assembly $(X,Y)$, the slice category $\cat E/X$ is a reflective subcategory of $\Asm(A',A)/(X,Y)$. For each $Z\in\cat E$ the slice $\cat E/Z$ is Cartesian closed because $\cat E$ is a topos, therefore $\Asm(A',A)$ is locally Cartesian closed.

\begin{lemma} \label{CC} Let $\cat E$ be a Heyting category, let $\cat D$ be a Cartesian closed full subcategory and let $L:\cat E\to\cat D$ be a finite limit preserving left adjoint to the inclusion of $\cat D$ into $\cat E$, such that the unit $\eta:L\to 1$ is a natural monomorphism. Then $\cat E$ is Cartesian closed.
\end{lemma}

\begin{proof} 
For simplicity, we will use the validity of first order logic and simply typed $\lambda$-calculus in the internal languages of respectively Heyting and Cartesian closed categories.

We define for all $Y,Z\in\cat E$
\[ Z^Y = \Set{f\in LZ^{LY}| \forall y\in Y.\exists z\in Z. f(\eta_Yy) = \eta_Zz} \]

For all $f:X\to Y^Z$, $x\in X$ and $y\in Y$, there exists a $z\in Z$ such that $f(x)(\eta_Yy) = \eta_Zz$ and because $\eta_Z$ is a monomorphism, this $z$ is unique. So let $f^t(x,y) = z$ if $f(x)(\eta_Yy) = \eta_Zz$ for all $x\in X$, $y\in Y$ and $z\in Z$.

For all $g:X\times Y\to Z$, $x\in X$ and $y\in Y$ we have $\eta_Z\circ g(x,y) = Lg(\eta_Xx,\eta_Yy)$. Note that we use $L(X\times Y) \simeq LX\times LY$ by the way. Because the subcategory is Cartesian closed, we can let $g^t(x) = \lambda y.Lg(\eta_Xx,\eta Yy)$.

For each $f:X\to Y^Z$, $x\in X$ and $y\in Y$ we have $(f^t)^t(x)(\eta_Yy) = \eta_Z z$ if and only if $f(x)(\eta_Yy) = \eta_Zz$. Therefore $(f^t)^t = f$. For each $g:X\times Y\to Z$, $x\in X$ and $y\in Y$ we have $(g^t)^t(x,y) = z$ if and only if $g^t(x)(\eta_Yy) = \eta_Zz$ while $g^t(x)(\eta_Yy) = \eta_Z\circ g(x,y)$. Since $\eta_Z$ is mono we have $(g^t)^t = g$. This means that $Z\mapsto Z^Y$ is right adjoint to $X\mapsto X\times Y$ and that $\cat E$ is Cartesian closed.
\end{proof}

\begin{lemma} \label{cleavage} For each $(X,Y)\in \Asm(A',A)$, there is a full and faithful functor $\cat E/X \to \Asm(A',A)/(X,Y)$ with finite limit preserving left adjoint. \end{lemma}

\begin{proof} The functor $\nabla$ is right adjoint to $\Dom$ and the unit of this adjunction $(X,Y) \to \nabla \Dom(X,Y)$ is a monomorphism. For each $(X,Y)\in\Asm(A',A)$, we let $\nabla_(X,Y): \cat E/X \to \Asm(A',A)/(X,Y)$ be the functor that maps $f:Z\to \Dom(X,Y)$ to $(\nabla f)^{-1}(Y)$: the subobject of $\nabla \Dom(X,Y)$ represented by $Y$. This functor is faithful and $\Dom$ acts as reflector $\Asm(A',A)/(X,Y) \to \cat E/X$ that preserves finite limits, and the unit is still a monomorphism.
\end{proof}

\begin{theorem} For each locally Cartesian closed Heyting category $\cat E$ and an OPCA pair $(A',A)$ in $\cat E$, the category of assemblies is a locally Cartesian closed Heyting category. \end{theorem}

\begin{proof} Lemma \ref{heyt} tells us $\Asm(A',A)$ is Heyting. For each assembly $(X,Y)$ lemma \ref{cleavage} embeds the Cartesian closed Heyting category $\cat E/\Dom(X,Y)$ into the Heyting category $\Asm(A',A)/(X,Y)$ in such way that the inclusion has a finite limit preserving left adjoint. Therefore every slice of $\Asm(A',A)$ is Cartesian closed according to lemma \ref{CC}, and that means $\Asm(A',A)$ is a locally Cartesian closed Heyting category.
\end{proof}

\subsection{Generic Monomorphisms}
We construct a generic monomorphism for the category of assemblies.

\begin{lemma} \label{genmon} Let $\cat E$ be a topos and $(A',A)$ an OPCA pair in $\cat E$. Let $D^*A\leq \Omega^A$ be the object of inhabited downward closed subobjects of $A$, and let $\set\in$ be the element-of relation $\Set{(a,U)\in A\times D^*A| a\in U }$. The inclusion $\id_{\Dom(D^*A,\set\in)}: (D^*A,\set\in) \to \nabla D^*A$ is a generic monomorphism.
\end{lemma}

\begin{proof} If $m:X\to Y$ is monic, then $\exists_m(X)\simeq X$. Therefore we can focus on monomorphisms of the form $\id_{\Dom(X,Y)}:(X,Y)\to (X,Y')$.

To $Y\leq A\times X$ belongs a characteristic map $y:X\to D^*A\leq \Omega^A$: $y(x) = \Set{a\in A| (a,x)\in Y}$, which by the definition of assemblies is a downward closed set. If we pull back $(D^*A,\set\in)$ along $y$ using the constructions from lemma \ref{heyt}, we get the assembly $(X,Y\land Y')$, where
\[ Y\land Y' = \Set{(a,x)\in A\times x| \forall t\in \true. at\converges\land (at,x)\in Y, \forall f\in \false. af\converges\land (af,x)\in Y' } \]
Since $Y\leq Y'$ we have $Y\land Y'\simeq Y$.
\end{proof}

\begin{theorem}  Let $\cat E$ be a topos and $(A',A)$ an OPCA pair in $\cat E$. The relative realizability topos $\RT(A',A) = \Asm(A',A)_{ex/reg}$ is a topos.
\end{theorem}

\begin{proof} The category of assemblies is locally Cartesian closed an has a generic mo\-no\-mor\-phism. This implies that its exact completion is a topos, according to Matias Menni \cite{Menni00exactcompletions}.
\end{proof}

\begin{remark}
Given any assembly $(X,Y)$ let $a:Y\to A$ and $x: Y\to X$ be the projections. Let $\Set\in = \Set{(a,\xi)\in A\times D^*A|a\in\xi }$, and let $b:\set\in \to A$ and $d:\set\in \to D^*A$ be the projections. There is a $y: X\to D^*A$ such that the square in the following commutative diagram is a pullback:
\[\xymatrix{
Y \ar[r]_{(a,y)} \ar@/^/[rr]^{a} \ar[d]_x & \set\in \ar[d]_{d} \ar[r]_{b} & A \\
X \ar[r]_y & D^*A
}\]
Because $\nabla$ and $\Asm(A',A)$ are regular and because of lemma \ref{gen1}, this means:
\[ (X,Y) \simeq \nabla y^{-1} \exists_{\nabla d} \nabla b^{-1}(\A)\]
So the generic monomorphism is the inclusion of $\exists_{\nabla d} \nabla b^{-1}(\A)$ into $\nabla D^*A$. Note the regular epimorphism $d:\nabla b^{-1}(\A) \to \exists_{\nabla d} \nabla b^{-1}(\A)$. It is a generic partitioned cover. If $\cat C$ is regular, $F:\Asm(A',A) \to \cat C$ preserves finite limits, $F\nabla$ is regular and $Fd$ is a regular epimorphism, then $F$ is a regular functor.
\end{remark}

\begin{example} For the OPCA pair $(\mathcal K_2^{\rm rec},\mathcal K_2)$ from example \ref{Ktwo}, a version of which exists in every topos with a natural number object, we now can construct the \emph{Kleene-Vesley topos} $\RT(\mathcal K_2^{\rm rec},\mathcal K_2)$ (see \cite{MR2479466}). This is a topos theoretic version of Kleene and Vesleys intuitionism in \cite{MR0176922}. The lattice of subterminal object is dual to the Medvedev lattice \cite{MR0073542} and has been studied as a model for constructive propositional logic, e.g., in \cite{MR2211183}. \end{example}

\subsection{Exact Completions} \label{contop}
In this section we recall the construction of the exact completion of a regular category. Using this construction we give a concrete description of the relative realizability topos.

\begin{definition} Given a regular category $\cat C$ let a \emph{subquotient} be a pair $(X,E)$ where $X\in\cat C$, $E\subseteq X^2$ and $E$ satisfies:
\[ (x,y)\in E \to (y,x)\in E\quad (x,y),(y,z)\in E \to (x,z)\in E \]
Given any two subquotients $(X,E)$ and $(X',E')$ and two subobjects $F,G\subseteq X\times X'$ let $F\simeq_{E\to E'} G$ if both
\begin{align*}
& (x,y)\in E \to \exists z\in X'. (z,z)\in E' \land (x,z)\in F \land (y,z)\in G, \\
& (x,x)\in E \land (x,y)\in F \land (x,z)\in G \to (y,z)\in E'
\end{align*}
If $F\subseteq X\times X'$ satisfies $F\simeq_{E\to E'} F$, then it is called a \emph{functional relation}. A \emph{morphism of subquotients} $(X,E)\to (X',E')$ is an equivalence class for $\simeq_{E\to E'}$.
\end{definition}

We explain how this definition works. For every subquotient $(X,E)$, the relation $E$ is symmetric and transitive in the internal language of $\cat C$. It defines an equivalence relation on $\Set{x\in X|(x,x)\in E}$. We use this pair to represent that quotient. The relations $\simeq_{E\to E'}$ are symmetric and transitive relation on the poset of subobjects of $X\times X'$. This defines an equivalence relation of an subset too, but this relation is external to $\cat C$. If $F\subseteq X\times X'$ and $F\simeq_{E,E'} F$, then $F$ induces a function form equivalence classes of $E$ to equivalence classes for $E'$. Therefore $F$ represents a morphism between quotients. If $G\subseteq X\times X'$, $G\simeq_{E,E'} G$ and $G\simeq_{E,E'} F$, then $G$ induces the same function as $F$. That is why morphisms $(X,E)\to (X',E')$ are equivalence classes for $\sim_{E\to E'}$.

\begin{lemma} Subquotients and morphisms for a regular category $\cat C$ together form a category $\cat C_{\rm ex/reg}$. This category $\cat C_{\rm ex/reg}$ is an exact completion of $\cat C$. \end{lemma}

\begin{proof} We compose relations $F\subseteq X\times Y$ and $G\subseteq Y\times Z$ by letting $G\circ F = \Set{(x,z)\in X\times Z| \exists y\in Y. (y,z)\in G, (x,y)\in F }$. If $F\simeq F'$ and $G\simeq G'$ relative to some subquotients, then $F\circ G \simeq F'\circ G'$. For every subquotient $(X,E)$ we have $E\simeq_{E,E} E$ and its equivalence class is an identity morphism.

The functor that sends each object $X$ to the pair $(X,\Delta_X)$, where $\Delta_X$ is the diagonal, and each arrow $f:X\to Y$ to the equivalence class of its graph, is an embedding of $\cat C$.

Finally, if $F:\cat C \to \cat D$ is a regular functor to an exact category, and $(X,E)$ is a subquotient, then $FE$ is an equivalence relation on a subobject of $X$. The subquotient $FX/FE$ exists here because of exactness, and that is where we map $(X,E)$ too. If $G\simeq G$ between $(X,E)$ and $(X',E')$, then composition with $FG$ induces a map $FX/FE \to FX'/FE'$. If $G\simeq G'$ then $FG$ and $FG'$ induce the same map. Thus $F$ factors through the category of subquotients in an up to isomorphism unique way.
\end{proof}

The inclusion $\id_{\Dom(X,Y)}:(X,Y)\to \nabla \Dom(X,Y)$ is a monomorphism in $\Asm(A',A)$. That means every assembly is a subobject of an object in the image of $\nabla$. In turn every subquotient is a subquotient of an object in the image of $\nabla$. If $m:(Y,E)\to\nabla X$ is a monomorphism that represents such a relation, then so does the isomorphic assembly $\im m(Y,E) \simeq (X^2,\im{\id_A\times m}(E))$. Therefore assemblies $(X^2,E)$ that define a subquotients of $\nabla X$ represent all objects of the relative realizability topos. We use these facts to get a simpler construction of relative realizability toposes.

\begin{definition} Let $\cat E$ be a topos and let $(A',A)$ be an OPCA pair. The standard relative realizability topos is defined as follows. The objects are pairs $(X,E\subseteq A\times X^2)$ such the the assembly $(X^2,E)$ is a symmetric and transitive relation on $\nabla X$. A morphism $(X,E) \to (X',E)$ is an isomorphism class of assemblies $(X\times X', Y)$, where $Y$ is a functional relation. \end{definition}

For each OPCA pair, the category of assemblies is the pseudoinitial regular model and the relative realizability topos is the pseudoinitial exact model. That is the main point of our paper. In the next section we explain some consequences of our definitions.

\section{Functors}\label{F}
In this section, we use initial models to find examples of regular functors from relative realizability categories into other categories. We no longer demand that the underlying category is a topos. However, when the underlying category is a topos many of the functors we construct have right adjoints and therefore are inverse images parts of geometric morphisms. For completeness we will also prove the existence of these right adjoints.

The first two subsections deal with geometric morphisms from localic toposes over the base category to relative realizability toposes. The examples we provide there are mostly new. More is known about morphisms between realizability toposes, which are the subject of the last subsection.

\subsection{Points} \label{points}
A \emph{point} of a topos $\cat T$ is a geometric morphism $\Cat{Set}\to\cat T$, where $\Cat{Set}$ is the topos of sets. The inverse image part of a geometric morphism is a regular functor, and this allows us to use our universal property. If $(A',A)$ is an OPCA pair in $\cat E$, then regular models $(F,C)$ where $F$ is a set valued regular functor represent each point of $\RT(A',A)$.

The analysis of set-valued regular models will have to wait for another paper. Here, we focus on regular models of the form $(\id_{\cat E},C)$ in stead. The reason that we hang on the to word `point', is that in the case that $\cat E = \Cat{Set}$, these models correspond to the class of all points of $\RT(A',A)$ that satisfy $f^*\nabla\simeq \id_{\Cat{Set}}$. As geometric morphisms, these are precisely the submorphisms of $\Dom\dashv \nabla:\cat E\to\RT(A',A)$.

For every Heyting category $\cat E$ the identity functor $\id_{\cat E}:\cat E\to\cat E$ is regular. If $(A',A)$ is an OPCA pair in $\cat E$, we can construct regular functors with filters of $A$. A \emph{$\id_{\cat E}$-filter} is a subobject $C$ of $A$ that satisfies:
\begin{itemize}
\item For all $x\in C$ and $y\in A$ if $y\geq x$ then $y\in C$.
\item For all $x,y\in C$ and $z\in A$ if $xy\convergesto z$ then $z\in C$.
\item For all $U\subseteq A$ if $U$ intersects $A'$ then $U$ intersects $C$.
\end{itemize}

\begin{remark} This last condition must be interpreted externally, not in the internal language of $\cat E$. An internal interpretation is possible if $\cat E$ is a topos, but that condition implies $A'\subseteq C$. \end{remark}

For each filter $C\subseteq A$, we construct the $C$-induced regular functor as follows. For each assembly $(X,Y)$ we let
\[ \Dom_C(X,Y) = \Set{ x\in X| \exists c\in C. (c,x)\in Y } \]
The functor then maps $f:(X,Y)\to (X',Y')$ to $\Dom f$ restricted to $\Dom_C(X,Y)$ factored through $\Dom_C(X',Y')$.

Quite surprisingly, all regular functors $G:\RT(A',A) \to \cat E$ that satisfy $G\nabla\simeq \id_{\cat E}$ are inverse image parts of geometric morphisms. Therefore relative realizability toposes can have many points, just like Grothendieck toposes.

\begin{theorem} Let $\cat E$ be a topos, let $(A',A)$ be an OPCA of $\cat E$ and let $C$ be a $\id_{\cat E}$-filter. Then the induced regular functor $\Dom_C:\RT(A',A)\to \cat E$ has a right adjoint. \end{theorem}

\begin{proof} We use the construction of the relative realizability topos form subsection \ref{contop} to get a clear picture of $\Dom_C:\RT(A',A)\to\cat E$. As $\Dom_C$ must preserve subquotients, we can construct the functor as follows. For each subquotient $(X,E)$, each $F:(X,E) \to (X',E')$ and each $\xi\in C(X,E)$ we let
\begin{align*}
\Dom_C(X,E) &= \Set{\xi\in \Omega^X | \exists x\in X. \xi = \Set{y\in X|\exists a\in C. (a,x,y)\in E} } \\
\Dom_CF(\xi) &= \Set{ y\in X'| \exists a\in C. (a,x,y)\in F }
\end{align*}

We now construct a functor $\nabla_C:\cat E \to \RT(A',A)$. For each $X\in \cat E$, let \[ E_X = \Set{ (a,f,g) \in A\times (\Omega^X)^2 | a\in C \to \exists x\in X. f=g=\set x} \] The assembly $(X^2,E_X)$ is a partial equivalence relation on $\nabla X$. For any arrow $f:X\to Y$ the morphism $\nabla f$ commutes with the partial equivalence relation of either side. Therefore we get a functor $\nabla_C$ by mapping each $X$ to $(X,E_X)$ and each $f:X\to Y$ to the morphism of subquotients it induces.

By computation we find that $\Dom_C\nabla_CX$ is isomorphic to $X$ for all $X\in\cat E$.
\[ CR_CX = \Set{ \xi\in \Omega^{\Omega^X} | \exists x\in X. \xi=\set{\set x}} \]
Let $e_X:\Dom_C\nabla_CX\to X$ be the inverse of $x\mapsto \set{\set x}$.

For each $(X,E)\in \RT(A',A)$ define $f_{(X,E)}: X\to \Omega^{\Dom_C(X,E)}$ by
\[ f(x) = \Set{\Set{y\in X| \exists a\in C. (a,x,y)\in E }} \]
If $(a,x,y)\in E$ and $a\in C$, then there is an $z\in \Dom_C(X,E)$  such that $f(x) = f(y) = \set z$, namely $z=\Set{y\in X| \exists a\in C. (a,x,y)\in E }$. So $(a,f(x),f(y))\in E_{\Dom_C(X,E)}$, and therefore $f$ is a morphism of the partial equivalence relations. Hence $f_{(X,E)}: (X,E) \to \nabla_C\Dom_C(X,E)$.

For $\xi \in \Dom_(X,E)$ we have $\Dom_Cf_{X,E}(\xi) = \Set{f(x)|x\in \xi} = \set{\set \xi}$. Therefore $e_{\Dom_C(X,E)}\circ \Dom_Cf_{(X,E)} = \id_C$. For $g\in \Omega^X$ we have $f_{\nabla_C X}(g) = \Set{\Set{g}}$. Therefore $\nabla_Ce\circ f_{\nabla_CX}  = \id_{\nabla_C}$. Hence we have an adjunction $\Dom_C\dashv \nabla_C$.
\end{proof}

\begin{remark} It is not clear that all geometric morphisms $f:\cat E \to\RT(A',A)$ satisfy $f^*\nabla\sim \id_{\cat E}$, even in the case that $\cat E =\Cat{Set}$. \end{remark}

\hide{\begin{remark} If $C$ and $C'$ are two $\id_{\cat E}$-filters such that $C\subseteq C'$, then $\Dom_C(X,Y)\subseteq \Dom_{C'}(X,Y)$ for all assemblies $(X,Y)$, and this inclusion is natural. In other words: the inclusion of $\id_{\cat E}$-filters corresponds to the specialization ordering of the geometric morphism, and $A'$ and $A$ represent bottom and top points for this ordering. \end{remark}}

\subsection{Characters} \newcommand{\psh}[1]{\cat E^{#1^{op}}}
We generalize the notion of point from the previous subsection. For each topos $\cat E$ and each OPCA pair $(A',A)$, we consider  geometric morphisms $\psh{P} \to \RT(A',A)$ where $P$ is a preordered object of $\cat E$, and $\psh{P}$ the topos of internal presheaves over $P$. Indirectly, we are looking at how toposes that are localic over $\cat E$ map into $\RT(A',A)$, because all localic toposes embed into a topos of the form $\psh{P}$.

The topos of internal presheaves $\psh{P}$ is constructed as follows. Each internal presheaf is an arrow $p: X \to P$ in $\cat E$, together with a restriction operator $r:\Set{(x,u)\in P\times X| x\leq p(u) } \to X$ that satisfies $p\circ r(x,u) = u$. Each morphism $f:(p,r) \to (p',r')$ is just an arrow $f: X\to X'$ such that $p'\circ f = p$ and $r'\circ f = f\circ r$.

The constant sheaf functor $\Delta:\cat E\to\psh{P}$ has both adjoints and is therefore a regular functor. Let $DP$ be the object of downsets of $P$. The $\Delta$-filters of $(A',A)$ correspond to arrows $A\to DP$.

\begin{definition} Let $P$ be a preordered set and $(A',A)$ and OPCA pair in $\cat E$. A character $\gamma$ is an arrow $A\to DP$ that satisfies:
\begin{itemize}
\item If $x \leq y$ then $\gamma(x)\leq\gamma(y)$.
\item If $xy\convergesto z$ then $\gamma(x)\cap\gamma(y)\leq \gamma(z)$.
\item If $a\in A'$ then $\gamma(a)=P$.
\end{itemize}
\end{definition}

We derive the next corollary from theorem \ref{main}.

\begin{corollary} Characters correspond to regular functors $\RT(A',A)\to \psh{P}$. \end{corollary}

\begin{proof} There is a bijection between $\cat E(A,DP)$ and  the subobjects of $\Delta A$:
\[ \cat E(A,DP)\simeq \cat E(A,\Gamma \Omega)\simeq \psh{DP}(\Delta A,\Omega)\simeq \Sub(\Delta A) \]
This bijection turns characters into $\Delta$-filters. \end{proof}

Because the functor $\gamma^*:\RT(A',A)\to\psh{P}$ is regular, we can give an explicit definition.
Let $(X,E)$ be any object and let for all $x\in X$ 
\[\db x_u = \Set{y\in X| \exists a\in A. u\in \gamma(a),(a,x,y)\in E} \]
Let $\gamma^*(X,E) = (X',p,r)$ with 
\[ X' = \Set{ (u, \xi) \in P\times \Omega^X | \exists x\in X.\xi=\db x_u } \]
$p:\gamma^*(X,E)$ is just the projection to the first coordinate. We let $r(\xi,u) = \bigcup_{x\in\xi} \db x_u$; $r$ now satisfies $r(\db x_u,v) = \db x_v$ for $v\leq u$.
Let $f:(X,E) \to (X',E')$ be any functional relation. Let for all $(u,\xi)\in \gamma^*(X,E)$
\[ \gamma_*f(u,\xi) = (u, \Set{ y\in X' | \exists a\in A,x\in \xi. u\in\gamma(a)\land (a,x,y)\in f }) \]

In the case that the underlying category is a topos, the functors that characters induce are not just regular, however.

\begin{theorem} Let $\cat E$ be a topos, $P$ a preordered set and $(A',A)$ an OPCA pair. Characters $A\to DP$ induce geometric morphisms $\psh{P}\to \RT(A',A)$. \end{theorem}

\begin{proof} For each object $(X,p,r)\in\psh{P}$, let 
\begin{align*}
x,y\in X\quad e(x,y) &= \Set{u\in P| u\leq px,u\leq py, r(x,u)=r(y,u)} \\
\gamma_*(X,p,r) &= (X, \Set{ (a,x,y)\in A\times X\times Y | \gamma(a) \subseteq e(x,y) }
\end{align*}
For each morphism $f:(X,p,r) \to (X',p',r')$ we let:
\[ \gamma_*f = \Set{(a,x,y) \in A\times X\times X' | \gamma(a) \subseteq e(fx,y) } \]

By writing out the definitions we find that if $\gamma^*\gamma_*(X,p,r)  = (X',p',r')$ then
\[ (u,\xi)\in X' \iff \exists x\in X. \xi = \Set{y\in X| u\leq p(y), r(y,u) = r(x,u) } \]
This new presheaf is isomorphic to $(X,p,r)$, by the following isomorphism:
\begin{align*}
g(x) &= (px,\Set{y\in X| u\leq p(y), r(y,u) = x }) \\
\forall x\in \xi. \quad \epsilon_{(X,p,r)}(u,\xi) &= r(x,u)
\end{align*}
The second family of morphisms acts as counit.

If $\gamma_*\gamma^*(X,E) = (X',E')$, then
\[ X' = \Set{ (u, \db x_u) \in P\times \Omega^X | x\in X } \]
with $\db x_u$ defined as before. We simplify the partial equivalence relation.
\[ E' = \Set{ (a,(u,\db x_u),(v,\db y_v)) | (\forall w\in \gamma(a). w\leq u,w\leq v)\land (a,x,y)\in E } \]
We define a family of functional relations $(X,E) \to \gamma_*\gamma^*(X,E)$ by
\[ \eta_{(X,E)} = \Set{ (a,x,(u,\db x_u)) \in A\times X\times X' | u\in \gamma(a) } \]
The subobject $\comb I$ tracks all of these, and together they form the unit.

We conclude that $\epsilon_{\gamma^*}\circ \gamma^*\eta = \id_{\gamma^*}$ because
\[ \gamma^*\eta_{(X,E)}(u,\xi) = (u, \Set{ (v,\db x_v)\in X' | x\in \xi })  = g(u,\xi) \]

By writing out definitions we also find that $(a,(u,\xi),y) \in \gamma_*\epsilon_{(X,p,r)}$ if for all $v\in \gamma(a)$ and $x\in \xi$, $v\leq u$ and $r(x,v) = r(y,v)$, while $(b,x,(u,\db x_u))\in \eta_{\gamma_*(X,E)}$ if $u\in \gamma(b)$. We have
$\gamma_*\epsilon\circ\eta_{\gamma_*} = \id_{\gamma_*}$, because for any $p\in\pair$ we have $\gamma(pab)=\gamma(a)\cap\gamma(b)$.

So $\gamma_*$ is right adjoint to $\gamma^*$. Since $\gamma^*$ is regular their combination is a geometric morphism.
\end{proof}

\begin{remark} An internal Grothendieck topology $J$ on a preordered object $P$ allows us to define a topos of sheaves $\Sh(P,J)$. This topos is embedded in $\psh{P}$ by a geometric morphism. Therefore, we can relate geometric morphisms $\Sh(P,J) \to \RT(A',A)$ to characters $\gamma:A\to DP$ of which the values are $J$-closed sets.\end{remark}

\begin{remark}
For the trivial poset that is the terminal object $\term$ we have $\psh{\term}\cong \cat E$, and characters are points.
\end{remark}

Toposes of sheaves are better understood then relative realizability toposes. By inducing geometric morphisms between these two kinds of toposes, characters may clarify the theory of relative realizability.

\subsection{Applicative Morphisms} \label{appmorph}
In this subsection we consider regular functors between realizability categories for different OPCA pairs. The filters that induce these functors are the applicative morphisms that were defined by Longley \cite{RTnLS}, Hofstra and van Oosten \cite{MR1981211} and Hofstra \cite{MR2265872}.

\begin{definition} Let $(A',A)$ and $(B',B)$ be two OPCA pairs in an Heyting category $\cat E$. An \emph{applicative morphism} $\gamma:(A',A) \to (B',B)$ is a $B$-assembly $(A,C)$ over $A$, such that the following subobjects of $B$ intersect $B'$.
\begin{align*}
&\Set{u\in B| \forall (x,y)\in C, y'\in A. y\leq y' \to (ux\converges \land (ux,y')\in C }\\
&\Set{r\in B| \forall (x',x), (y',y)\in C. xy\converges \to ((rx')y'\converges \land ((rx')y',xy)\in C) } \\
\forall a\in A'\quad & \Set{b\in B| (b,a)\in C }
\end{align*}
\end{definition}

\begin{theorem} For each applicative morphism $\gamma:(A',A) \to (B',B)$ there is an up to isomorphism unique regular functor $F: \Asm(A',A) \to\Asm(B',B)$ such that $F\A \simeq (A,C)$ and $F\nabla\simeq \nabla$. For each regular functor $F:\Asm(A',A) \to\Asm(B',B)$ such that $F\nabla\simeq \nabla$, there is an up to isomorphism unique applicative morphism $\gamma:(A',A) \to (B',B)$.\end{theorem}

\begin{proof} $\gamma$ is a filter for $\nabla:\cat E\to \Asm(B',B)$, so $(\nabla,\gamma)$ is a regular model for $(A',A)$. Therefore there is an up to isomorphisms unique regular functor $\Asm(A',A)\to\Asm(B',B)$ satisfying the conditions.

Any regular functor $F$ such that $F\nabla \simeq \nabla$ will map $\id_A:\A\to\nabla A$ to some monomorphism $F\A\to F\nabla A$. The image of $F\A$ along the composition of $F\id_A$ with the isomorphism $F\nabla A \to \nabla A$ is an applicative morphism because $F$ preserves filters.
\end{proof}

Unlike characters, applicative morphisms do not generally induce geometric morphisms if the underlying category is a topos. The ones that do have the following property.

\begin{definition} For $\gamma:(A',A)\to(B',B)$ we define the arrow $\gamma:A\to DB$ by $\gamma(a) = \Set{b\in B|b\in\gamma(a)}$. We define the following relation on $DB$: $UV \convergesto W$ if and only if 
\[ \forall x\in U,y\in V.\exists z\in W. xy\convergesto z \]
The term $UV$ stands for the least $W\in DB$ such that $UV\convergesto W$ and remains undefined if no such $W$ exists. The applicative morphism $\gamma$ is \emph{computationally dense} if there is some $\mu\subseteq B$ intersecting $B'$ such that for each $U\in DB$ that intersects $B'$ the following subobject of $A$ intersects $A'$.
\[ U^\mu = \Set{a\in A| \forall x\in A. U\gamma(x)\converges \to ax\converges\land \mu\gamma(ax) \convergesto U\gamma(x)\converges)} \]
\end{definition}

\begin{theorem} Computationally dense applicative morphisms induce geometric morphisms between relative realizability toposes. \end{theorem}

\begin{proof} We leave to the reader to check that for each relative realizability topos $\RT(A',A)$ over a base topos $\cat E$ the assignment $X\to \Sub(\nabla -)$ is a tripos over $\cat E$ and that an adjoint pair of transformations of triposes induces a geometric morphism \cite{MR2479466}.

For clarity, let $(\nabla_A,\A)$ be an initial exact model for $(A',A)$ and $(\nabla_B,\mathring B)$ for $(B',B)$. Since the regular functor that $\gamma$ preserves $\nabla$ and subobjects, the functor relates to a transformation of triposes $\Sub(\nabla_A-) \to \Sub(\nabla_B-)$. So we need to find a right adjoint to that transformation.

Fixing $X\in \cat E$, we may represent subobjects of $\nabla_AX$ by subobjects of $A\times X$ and subobjects of $\nabla_B$ by subobjects of $B\times X$. We can represent the transformation induced by $\gamma=(A,C)$ with the following map.
\[ \gamma^*Y = \Set{(b,x)\in B\times X| \exists a\in A. (b,a)\in C\land (a,x)\in Y } \]
Now we finally start constructing a right adjoint.
\[ \gamma_\mu Y =  \Set{ (a,x)\in A\times X| \mu \gamma(a)\downarrow\Set{b\in B| (b,x)\in Y} } \]
Automatically $\mu$ tracks the inclusion $\id_{\Dom(X,\gamma^*\gamma_\mu Y)}:(X,\gamma^*\gamma_\mu Y) \to (X,Y)$. To find a tracking for the inclusion $(X,Y)\to (\gamma_\mu\gamma^*Y)$ let
\[ \iota = \Set{b\in B| \forall x\in B.\exists y\leq x. bx\convergesto y } \]
Since the identity arrow is combinatory, the subobject $\iota$ intersects $B'$ and $\iota^\mu$ intersects $A'$. The tracking we need is $\iota^\mu$.

To establish that $\gamma_\mu$ is a well defined mapping $\Sub(\nabla_BX) \to \Sub(\nabla_A X)$, let $(X,Y)$ and $(X',Y')$ be any pair of assemblies for $(B',B)$, and let \[U = \Set{b\in B| \forall (x,y)\in Y. bx\converges, (bx,y)\in Y'} \] If $(a,x)\in \gamma_\mu(Y)$ and $u\in U^\mu$, then $ua\converges$ and $\mu\gamma(ua)\downarrow U\gamma(a)$. This implies $U^\mu$ tracks the inclusion of $(X,\gamma_\mu(Y))$ into $(X,\gamma_\mu(Y'))$.

Thus we get a right adjoint to $\gamma^*$, and a geometric morphism of relative realizability toposes.
\end{proof}

\section{Conclusion}
The relative realizability topos for an OPCA pair $(A',A)$ in $\Cat{Set}$ satisfies a universal property: $\RT(A',A)$ is the universal exact category that adds a new subobject to $A$ that is closed under application and that intersects all subsets that intersect $A'$, while preserving regular propositions. There is a construction for relative realizability categories for OPCA pairs in other Heyting categories that satisfies a similar universal property. The universal property allows us to study regular functors by studying filters of order partial combinatory algebras.

I thank the referees for their useful remarks.

\subsection{Further Thoughts}
We consider a couple of topics for future publications.

Carboni and Celia Magno \cite{MR678508} described the exact completion of left exact categories. Robinson and Rosolini showed \cite{MR1056382} that realizability toposes constructed over the category of sets are exact completions of subcategory of partitioned assemblies. Carboni noted \cite{MR1358759} that the category of assemblies is an intermediate step, being the \emph{regular completion} of the category of partitioned assemblies. The relation between the various completions is explained in \cite{MR1600009}.

Relative realizability toposes over toposes where epimorphisms don't split no longer are exact completions of their categories of partitioned assemblies. Hofstra developed the alternative notion of \emph{relative completion} \cite{MR2067191}, to deal with the more general case. Relative completions works for OPCA pairs $(A',A)$ where $A'$ has enough global sections, which means that every inhabited subobject has a global lower bound in $A'$. It may be interesting to see if there is a natural completion construction that works for other pairs.

While the limitations of Heyting categories require the universal property we gave in this paper, it is possible to characterize relative realizability \emph{toposes} by another, possibly more useful pseudoinitiality property. In a topos $\cat E$ every OPCA pair $(A',A)$ has a `completion' $(B',B)$, where $B=D^*A$ (see lemma \ref{genmon}) and $B'\subseteq D^*A$ is the object of downsets of $A$ that intersect $A'$. We call this a `completion', because $D^*A$ is closed under joins of inhabited subobjects. Because of this completeness property, every representable function is globally representable: any representable $f:B^n\partar B$ is represented by the global section $\bigcup\db f:\term \to B$, because application in $B$ preserves joins.

We believe that relative realizability toposes, and some of their subtoposes, have a the following universal property involving complete OPCA pair and left exact functors. Let a \emph{left exact model} be the combination of a left exact $F:\cat E\to\cat C$ with a $C\subseteq FB$ that is upward closed and closed under application, and through which every $Fx:\term \to FB'\subseteq FB$ factors. The functor $\nabla:\cat E\to \Asm(A',A)$ together with $\mathring B = (B,\set{(a,b)\in A\times B| a\in b})$ is a pseudoinitial left exact model. Using Frey's analysis of the tripos-to-topos construction \cite{a2CAotTtTC}, we may be able to derive another universal property of $\RT(A',A)$.

In \cite{MR2265872}, Hofstra uses \emph{basic combinatory objects} to provide a framework for all kinds of realizability.
Complete OPCA pairs, like $(B',B)$, are a special case of basic combinatory objects, and this may help research in this area.


\end{document}